\documentclass{birkjour}

\usepackage{amsmath}
\usepackage{amsthm}
\usepackage{amssymb}
\usepackage{amsfonts}
\usepackage{amsxtra}  
\usepackage{epsfig}
\usepackage{verbatim}

\newtheorem{theorem}{Theorem}[section]
 \newtheorem{corollary}[theorem]{Corollary}
 
 \newtheorem{proposition}[theorem]{Proposition}
 
 \theoremstyle{definition}

 \theoremstyle{remark}
 
 \numberwithin{equation}{section}

\begin{document}

%-------------------------------------------------------------------------
% editorial commands: to be inserted by the editorial office
%
%\firstpage{1} \volume{228} \Copyrightyear{2004} \DOI{003-0001}
%\seriesextraline{This is the Concrete Title of this Book\br H.E. R and S.T.C. W, Eds.}
%\firstpage{1}

\title[Bergman projection]{Estimates of the $L^p$ norms of the Bergman projection on strongly pseudoconvex domains}
\author{ \v Zeljko \v Cu\v ckovi\' c }
\begin{abstract}
We give estimates of the $L^p$ norm of the Bergman projection on a strongly pseudoconvex domain in $\mathbb{C}^n$.  We show that this norm is comparable to $\frac{p^2}{p - 1}$ for $1 <p< \infty$.
\end{abstract}
\address{Department of Mathematics and Statistics, \newline University of Toledo, Toledo,
Ohio, 43606 \newline zcuckovi@math.utoledo.edu}

\maketitle

\section{Introduction}

Let $\Omega\subset \mathbb{C}^n$ be a bounded domain and let $H(\Omega)$
denote the holomorphic functions on $\Omega$.  Let $L^2(\Omega)$ be the standard space of square-integrable functions with respect to Lebesgue measure with the usual $L^2$ inner product. The Bergman space is defined as $ A^2 (\Omega) =:H(\Omega)\cap L^2(\Omega)$.
The Bergman projection is the orthogonal projection operator $P:L^2(\Omega)\longrightarrow A^2(\Omega).$ This operator is one of the most fundamental operators in complex analysis.
The Bergman kernel function
$K(z,w)$ defined on $\Omega \times \Omega$ represents the Bergman
projection as an integral operator

$$Pf(z)=\int_\Omega K(z,w)f(w)\, dw,\qquad f\in L^2(\Omega)$$
where $dw$ denotes integration in the $w$ variables, with respect to
the euclidean volume form. 

In this paper we are interested in estimating the $L^p$ norm of the Bergman projection on strongly pseudoconvex domains.

Suppose that $\Omega$ is smoothly bounded, that is: there exists a $C^\infty$,
real-valued function $r:\,\,\text{nbhd}
\left(\overline\Omega\right)\longrightarrow\mathbb{R}$ such that
$\Omega=\{ z: r(z) <0\}$ and $dr\neq 0$ when $r=0$. If $\Omega$ is strongly 
pseudoconvex, i.e. $i\partial\bar\partial r(p)\left(\xi,\bar\xi\right) >0$ for all $p\in b\Omega$
and all vectors $\xi\in\mathbb{C}^n$ 
satisfying $\partial r(p)\left(\xi\right)=0$, the
boundary behavior of the Bergman kernel function associated to $\Omega$ is 
understood quite precisely.

For this class of domains, the mapping properties of $P$ in many classical Banach
spaces have been established.  For our purposes, we mention the classical result that $P$ is a bounded operator on $L^p(\Omega)$ for $1 < p < \infty$ \cite{Pho-Ste}.

On strongly pseudoconvex domains, Fefferman \cite{Fef} has established a complete
asymptotic expansion of $K(z,w)$, in terms of $r(z), r(w)$ and a pseudo-distance between
$z$ and $w$, as $z,w\to b\Omega$ (see also \cite{BoM-Sjo}).  A crucial feature of the Bergman kernel on a strongly pseudoconvex domain is that its singularities occur only on the boundary diagonal, instead of on the full diagonal in  $\Omega\times\Omega$. 

The simplest example of a strongly pseudoconvex domain in $ \mathbb{C}^n$ is an open unit ball $\mathbb{B}_n$.  An important precursor to our work is the result by Zhu who established a sharp estimate of the $L^p$ norm of $\|P\|$ \cite{Zhu}.  We state the unweighted version of his main theorem.

\begin{theorem}[Zhu]\label{TheoremA} 
For all $1 < p < \infty$ there exists a constant $C > 0$, depending on $n$ but not on $p$, such that the norm of $P:L^p(\mathbb{B}_n)\longrightarrow A^p(\mathbb{B}_n)$  satisfies the estimate 

$$C^{-1} \csc \frac{\pi}{p} \leq \|P\|_p \leq C \csc \frac{\pi}{p}.$$

\end{theorem}

Zhu has also restated his results in the following way:  there exists a constant $C > 0$ independent of $p$ such that 

$$C^{-1} \frac{p^2}{p-1} \leq \|P\|_p \leq C \frac{p^2}{p-1}.$$

We will prove estimates of $\|P\|_p$  for $P:L^p(\Omega)\longrightarrow A^p(\Omega),$ in case $\Omega$ is strongly pseudoconvex. The first theorem gives the upper estimate that is analogous to the estimate of Zhu.

\begin{theorem}\label{T:Theorem1} 
Let $\Omega$ be a smoothly bounded strongly pseudoconvex domain.  For all $1 < p < \infty$ there exists a constant $C' > 0$, depending on $n$ but not on $p$, such that the norm of $P:L^p(\Omega)\longrightarrow A^p(\Omega)$  satisfies the estimate 

$$ \|P\|_p \leq C' \frac{p^2}{p-1}.$$

\end{theorem}

The next theorem gives lower estimates for the norm of $P$ for more general domains than strongly pseudoconvex domains.  Recall that 
a domain $\Omega$ satisfies the Condition R if $P$ maps the space $C^{\infty} 
(\overline\Omega)$ into itself. 

\begin{theorem}\label{T:Theorem2}
Let $\Omega\subset \mathbb{C}^n$ be a smoothly bounded domain and assume that the Condition R holds on $\Omega$.
Then there exists $2 < p_0 < \infty$ and a constant $C' > 0$, depending on $n$ but not on $p$, such that 

$$C' \frac{p^2}{p-1} \leq \|P\|_p $$
for all $p \in (1, q_0) \cup (p_0, \infty)$,  with $\frac 1p_0 +\frac 1q_0 =1.$

\end{theorem}

Since our domain $\Omega$ is strongly pseudoconvex, it does satisfy the Condition R.  The following corollary follows immediately.

\begin{corollary}\label{T:Corollary 3}
	Let $\Omega\subset \mathbb{C}^n$ be a smoothly bounded strongly pseudoconvex domain.
Then there exists $2 < p_0 < \infty$ and a constant $C' > 0$, depending on $n$ but not on $p$, such that 

$$C' \frac{p^2}{p-1} \leq \|P\|_p $$
for all $p \in (1, q_0) \cup (p_0, \infty)$,  with $\frac 1p_0 +\frac 1q_0 =1.$

\end{corollary}
Notice that the lower estimates are obtained for $p$ sufficiently large (or
sufficiently close to 1). We do not know if there is an underlying
reason for this restriction. 

The advantage of working on the open unit ball is that an explicit formula for the Bergman kernel is known. On strongly pseudoconvex domains we only have estimates for the Bergman kernel and it seems that the type of results in Theorem \ref{T:Theorem1} and \ref{T:Theorem2} has not been considered before. 

The paper is laid out as follows. In Section 2 we prove the upper estimate on $\|P\|_p$, while
in Section 3 we prove the lower estimate.

{\bf Acknowledgment.} 
The author would like to thank Sonmez Sahutoglu for several discussions during the preparation of this work and to Akaki
Tikaradze for pointing out a geometric argument in the lower estimate.
 
\section{Upper estimate}

The work in this section was motivated by the work \cite{McN-Ste} and the related work on Toeplitz operators in \cite{Cuc-McN}.  The proof of our upper estimate depends on the Schur lemma applied to the Bergman kernel.

\begin{proposition}\label{P:Prop2.1} 
Suppose $\mu$ is a positive measure on a space $X$ and $K(x, y)$ is a nonnegative measurable function on $X \times X$.
Let $1< p <\infty$ be given and let $q$ be the conjugate exponent
of $p$, $\frac 1p +\frac 1q =1$.  Suppose there exists $C>0$ and a positive function $h(x)$ on $X$ such that
\begin{equation}\label{E:2.2}
  \int_X K(x,y) h(y)^q d \mu(y) \leq
  C h(x)^q%\tag 2.8$$  
\end{equation}
for almost all $x$ in $X$ and
\begin{equation}\label{E:2.3}
  \int_X K(x,y) h(x)^p d \mu(x) \leq
  Ch(y)^p%\tag 2.8$$
\end{equation} 
for almost all $y$ in $X$. 
Then the integral operator
\begin{equation}\label{E:2.4}
Tf(x)=\int_X K(x,y)f(y)\, d \mu(y)%\tag 2.9$$
\end{equation}
is bounded on $L^p(X, \mu)$ with norm not exceeding $C$. 
\end{proposition}

If $\Omega$ is smoothly bounded strongly pseudoconvex domain,  the earlier work in \cite{Hor} contained only estimates on the Bergman kernel restricted to the diagonal.  As we have already mentioned Fefferman \cite{Fef} has established a complete
asymptotic expansion of $K(z,w)$ on $\Omega$. For our purposes we now recall only the upper bounds on the Bergman kernel. 

\begin{proposition}\label{P:Prop2.5}
 Let $\Omega=\{ r<0\}$ be a smooth, bounded, strongly pseudoconvex
domain in $\mathbb{C}^n$. Let $K(z, w)$ denote the Bergman kernel of $\Omega.$ For each $p\in b\Omega$, there exists a neighborhood $U$ of $p$, holomorphic coordinates $(\zeta_1,\dots ,\zeta_N)$ and a constant $C>0$, such that
for $z,w\in U\cap\Omega$
\begin{equation}\label{E:2.6}
 |K(z,w)|\leq C\left( |r(z)|+|r(w)|+ |z_1- w_1| +\sum_{k=2}^n 
 |z_k -w_k|^2\right)^{-(n+1)}.%\tag 3.1$$
\end{equation}
Here $z=(z_1,\dots , z_n)$ in the $\zeta$-coordinates, and similarly for $w$.
\end{proposition}

Inequality (\ref{E:2.6}) can be extracted from the results in \cite{Fef}; it may also 
be obtained by simpler methods as shown by \cite{McN} and \cite{N-R-S-W}, see the remark (5.3) in \cite{McN}.

In what follows we use the notation $f(z)\lesssim g(z)$ to denote that there
exists a constant $C$, independent of $z$ and $\epsilon$, such that $f(z)\leq Cg(z)$.

In order to apply the Schur lemma, we need the following proposition.
  
\begin{proposition}\label{P:Prop2.7} Let $\Omega=\{ r<0\}$ be a smooth, bounded, strongly
pseudoconvex domain in $\mathbb{C}^n$, and let $K(z,w)$ be the 
Bergman kernel associated to $\Omega$. 

Then 
\begin{equation}\label{E:2.8}
 \int_\Omega \left| K(z,w)\right||r(w)|^{-\epsilon}\, dw
 \lesssim  \frac{1}{\epsilon (1 - \epsilon)} |r(z)|^{-\epsilon},%\tag 3.3$$
\end{equation}
for all $0 <\epsilon <1$.
\end{proposition}

\begin{proof} 
The proof follows the lines of Lemma 1 in \cite{McN-Ste}.  We will skip certain steps to avoid the repetition.  
Let $\Delta_b=\{ (z,z): z\in b\Omega\}$ be the boundary diagonal
of $\overline\Omega\times\overline\Omega$. It is well known that
\begin{equation}\label{E:2.9}
  K(z,w)\in C^\infty\left( \overline\Omega\times\overline\Omega\setminus
  \Delta_b\right)%\tag 3.4$$
\end{equation}
see \cite{Ker}.

Cover $b\Omega$ by neighborhoods $U_1,\dots U_M$ given by
Proposition \ref{P:Prop2.5}; we may assume that the neighborhoods are so small that
the quantity in parenthesis on the right hand side of (\ref{E:2.6}) is less than 1. 

Now consider an arbitrary $U_j$, $1\leq j\leq M$ and let $z,w\in
\overline\Omega\cap U_j$.

Assume $z\in U_j$ is temporarily fixed. Then Proposition \ref{P:Prop2.5} gives
\begin{equation*}
\begin{split}
 I_j&=\int_{U_j}\left| K(z,w)\right||r(w)|^{-\epsilon}\, dw
\\ &\leq C\int_{\mathbb{C}^n}
\left( |r(z)|+|r(w)|+ |z_1- w_1| +\sum_{k=2}^n 
|z_k -w_k|^2\right)^{-(n+1)} |r(w)|^{-\epsilon}\, dw.
\end{split}
\end{equation*}

As in \cite{McN-Ste}, we change the coordinates: $\widetilde w_k = w_k- z_k$, $k=2,\dots , n$, 
$\text{Re }
\widetilde w_1 = r(w)$, $\text{Im }\widetilde w_1 = \text{Im } w_1$. Now let $x=\text{Re 
}\widetilde w_1$ and $y=\text{Im }z_1
-\text{Im } w_1$. We then obtain 

\begin{equation}\label{E:2.10}
I_j\leq C\int_{\mathbb{C}^n}\left( |r(z)|+|x|+ |y| +\sum_{k=2}^n 
|\widetilde w_k|^2\right)^{-(n+1)} |x|^{-\epsilon}\, d\widetilde w_2\dots d\widetilde w_n\,
dx\, dy.%\tag 3.6$$
\end{equation}

We start evaluating this iterated integral by performing the $\widetilde w_2$ integration first.
Define

\begin{eqnarray*}
 R_1 &=&\left\{ \widetilde w_2: |\widetilde w_2|^2 >
|r(z)|+|x|+ |y| +\sum_{k=3}^n 
|\widetilde w_k|^2\right\} \\
R_2 &= &\left\{ \widetilde w_2: |\widetilde w_2|^2 <
|r(z)|+|x|+ |y| +\sum_{k=3}^n 
|\widetilde w_k|^2\right\}.
\end{eqnarray*}
Using polar coordinates on the region $R_1$ we have
\begin{equation*}
\begin{split}
 \int_{R_1} &\left( |r(z)|+|x|+ |y| +\sum_{k=2}^n 
|\widetilde w_k|^2\right)^{-(n+1)} |x|^{-\epsilon}\, d\widetilde w_2 \\
&\leq \int_{R_1} \left(|\widetilde w_2|^2\right)^{-(n+1)} |x|^{-\epsilon}\, d\widetilde w_2 \\
&\lesssim \left(|r(z)|+|x|+ |y| +\sum_{k=3}^n 
|\widetilde w_k|^2\right)^{-n}|x|^{-\epsilon}.
\end{split}
\end{equation*}
On the region $R_2$ we obtain the same upper bound by using the estimate
\begin{equation*}
\begin{split}
 \int_{R_2} &\left( |r(z)|+|x|+ |y| +\sum_{k=2}^n 
|\widetilde w_k|^2\right)^{-(n+1)} |x|^{-\epsilon}\, d\widetilde w_2  \\
&\leq\left(|r(z)|+|x|+ |y| +\sum_{k=3}^n 
|\widetilde w_k|^2\right)^{-(n+1)}|x|^{-\epsilon}\,\text{vol}(R_2). \\
\end{split}
\end{equation*}

We continue in the same way to perform the integration on $d\widetilde w_3, \dots
d\widetilde w_n$ and $dy$ integrals, reducing one negative power
of the integrand at each step, to obtain
$$I_j\lesssim \int_\mathbb{ R}\left(|r(z)| +|x|\right)^{-1}|x|^{-\epsilon}\, dx.$$

We now estimate this final integral:

\begin{equation*}
\begin{split}
 2 \int_ 0^{\infty}&\left( |r(z)|+x 
\right)^{-1} x^{-\epsilon}\, dx\\
&=2 \int_{|r(z)|}^{\infty} \left( |r(z)|+x 
\right)^{-1} x^{-\epsilon}\, dx + 2 \int_{0}^{|r(z)|} \left( |r(z)|+x 
\right)^{-1} x^{-\epsilon}\, dx  \\
&\leq 2 \int_{|r(z)|}^{\infty}  x^{-\epsilon - 1}\, dx + 2 \int_{0}^{|r(z)|} |r(z)|^{-1} x^{-\epsilon}\, dx  \\
& = \frac{1}{\epsilon} |r(z)|^{-\epsilon} +  \frac{1}{1 - \epsilon} |r(z)|^{-\epsilon} \\
&= \frac{1}{\epsilon (1 - \epsilon)} |r(z)|^{-\epsilon}. \\
\end{split}
\end{equation*}
Thus we obtain 

$$I_j\lesssim \frac{1}{\epsilon (1 - \epsilon)} |r(z)|^{-\epsilon}.$$

If $U_0=\overline\Omega\setminus\cup_{j=1}^M U_j$ then,
it follows from (\ref{E:2.9}) that there exists $M > 0$ so that $|K(z, w)| \leq M$ for all $z, w \in U_0$.  Similarly for $w \in U_0$, there exists $\eta_1$ and $\eta_2$ such that $ \eta_1 \leq |r(w)| \leq \eta_2$ for all $w \in U_0$. Hence 

\begin{equation*}
\begin{split}
\int_{U_0} \left| K(z,w)\right||r(w)|^{-\epsilon}\, dw
&\lesssim \eta_1^{-\epsilon} \\
& \lesssim (\frac{\eta_2}{\eta_1})^{\epsilon} |r(z)|^{-\epsilon}\\
& \lesssim |r(z)|^{-\epsilon}\\
\end{split}
\end{equation*}
since $\frac{\eta_2}{\eta_1} > 1$ and $0 < \epsilon < 1$.

Thus we obtained the same bounds on each of the open sets $U_0, \dots ,
U_M$, we have shown (\ref{E:2.8}).  
\end{proof}

We are now prepared to give the proof of the upper estimate.

\begin{proof}[Proof of the upper estimate in Theorem \ref{T:Theorem1}]
We apply the Schur lemma to the function $h(z) = |r(z)|^{-\frac{1}{pq}}.$
Then (\ref{E:2.2}) becomes

\begin{equation*}
\int_\Omega |K(z,w)|h(w)^q dw = \int_\Omega |K(z,w)||r(z)|^{-\frac{1}{p}} dw.\\
\end{equation*}

Now apply Proposition 2.7 with $\epsilon = \frac{1}{p}$ to obtain for all $z \in {\Omega}$ 

\begin{equation*}
\begin{split}
\int_\Omega |K(z,w)|h(w)^q dw
&\lesssim  pq |r(z)|^{-\frac{1}{p}}\\
&= \frac{p^2}{p - 1}h(z)^q.\\
\end{split}
\end{equation*} 

By the symmetry of the kernel and the estimate $pq$ on the right hand side, (\ref{E:2.3}) is satisfied for all $w \in {\Omega}$ with the same upper bound.  We conclude that 
$$\|P\|_p \lesssim \frac{p^2}{p-1}.\qedhere$$
\end{proof}

\section{Lower estimate}

After doing a holomorphic change of coordinates, for simplicity, we can assume that $0$ is in the boundary of $\Omega$ and $y_n = \text{Im}z_n > 0$ is the outward real normal. Let us consider the case $p >2.$  Now for each $p$ choose $z_p$ to be the point $z_p = (0, 0, \dots, -ie^{-p}).$ Then $z_p \in \Omega$ for $p$ large enough.

Now we follow the idea from \cite{Zhu}.  Let 
$$f(z) = \log (iz_n) - \overline{\log(iz_n)}.$$  Clearly $f(z) = 2i \arg(iz_n)$ is a bounded function on 
$\Omega$ and hence $\|f\|_p \leq 2\pi.$ 

It is not difficult to see that $\log (iz_n) \in A^2(\Omega)$. Hence $Pf(z) = \log (iz_n) - 
P(\overline{\log(iz_n)}).$ Now we recall a recent result from \cite{HMS}, Corollary 1.12. 

\begin{theorem}[Herbig-McNeal-Straube]\label{T:Theorem3.1}
Let $\Omega\subset \mathbb{C}^n$ be a smoothly bounded domain and assume that the Condition R holds on $\Omega$. Then for any $f \in A^2(\Omega)$, $P \overline f$ is smooth up to the boundary.
\end{theorem}

We now use this theorem to conclude that there exists a constant $M > 0$ so that $|P(\overline{\log(iz_n)})| < 
M$ on $\Omega$.

It is well known that point evaluations are bounded on $A^p(\Omega)$, see for ex. \cite{Kr}, Lemma 1.4.1 modified for arbitrary $p > 1$ or \cite{Dur-Sch} in case of planar domains.  More concretely, for a $z\in {\Omega}$ we have 

$$|f(z)| \leq C(n) d(z)^{-\frac{2n}{p}} \|f\|_p$$
for all $f \in A^p(\Omega)$, where $d(z)$ denotes the distance of $z$ to the boundary.  We will apply this inequality to $f(z)$ with the point $z_p$ chosen above.  Hence we obtain
\begin{align}\label{E:3.4}
\nonumber \|Pf\|_p &\geq C(n) d(z_p)^{\frac{2n}{p}} |Pf(z_p)|\\
& \gtrsim (e^{-2n}) |Pf(z_p)|\\
\nonumber & \gtrsim \left| |\log(e^{-p})| - |P(\overline{\log(iz_n))(e^{-p})}|\right| \\
\nonumber & > p - M \\
\nonumber & \gtrsim p
\end{align}
for $p$ large enough.  Hence we conclude that there exists a $p_0 > 2$ such that $\|Pf\|_p \gtrsim p$ for all $p > p_0$.  This shows that 
$$\frac{\|Pf\|_p}{\|f\|_p} \gtrsim p$$
for $p > p_0$ which shows that $\|Pf\|_p > Cp$, with $C$ depending on $n$ but not on $p > p_0$.  Since $p > 2$, we have $p > \frac{1}{2} \frac{p^2}{p-1}$ which gives 
$$\|Pf\|_p > C'\frac{p^2}{p-1},$$
with $C'$ depending on $n$ but not on $p > p_0$.  This gives the lower estimate in our theorem.  
By duality and the symmetry of $pq$ we get the estimate for the range $p \in (1, q_0)$.


\begin{thebibliography}{99}

\bibitem{BoM-Sjo}
L. Boutet de Monvel \& J. Sj\"ostrand,
\emph{Sur la
singularit\'e des noyaux de Bergman et de Szeg\"o},
Soc. Math. France
Ast\'erisque
\textbf{34-35} (1976), 123-164

\bibitem{Cuc-McN}
 \v Z. \v Cu\v ckovi\' c \& J. D. McNeal,
 \emph{Special Toeplitz operators on strongly pseudoconvex domains},
Rev. Mat. Iberoamericana \textbf{22} (2006), 851-866

\bibitem{Dur-Sch}
P. Duren \& A. Schuster,
\emph{Bergman spaces},
Math. Surveys and Monographs \textbf{100},  Amer. Math. Soc. (2004)

\bibitem{Fef}
C. Fefferman,
\emph{The Bergman kernel and biholomorphic
mappings of pseudoconvex domains}, 
Inv. Math. \textbf{26}  (1974), 1-65

\bibitem{HMS}
 A.-K.Herbig, J.D. McNeal \&  E.J. Straube,
\emph{Duality of holomorphic function spaces and smoothing properties of the
Bergman projection}, 
Trans. Amer. Math. Soc. \textbf{366}  (2014), 647-665

\bibitem{Hor}
L. H\"ormander,
\emph{$L^2$ estimates and existence theorems for the
$\overline\partial$-operator},
Acta Math. \textbf{113} (1965),   89-152
 

\bibitem{Ker}
N. Kerzman,
 \emph{The Bergman kernel function. Differentiability at the
boundary},  Math. Ann. \textbf{195}  (1972),   149-158

\bibitem{Kr}
S. G. Krantz
\emph{Function theory of several complex variables}, Sec. Ed.,
 AMS Chelsea Publishing 2001.

\bibitem{McN}
J.D. McNeal,
\emph{Boundary behavior of the Bergman kernel
function in $\mathbb{C}^2$},
Duke Math. J. \textbf{58}  (1989), 499-512


\bibitem{McN-Ste}
 J.D. McNeal \&  E.M. Stein,
\emph{Mapping properties of the
Bergman
projection on convex domains of finite type}, 
Duke Math. J. \textbf{73}  (1994), 177-199

\bibitem{N-R-S-W}
A. Nagel, J.P. Rosay, E.M. Stein, \& S. Wainger,
\emph{Estimates for the Bergman and Szeg\"o kernels in $\mathbb{C}^2$},
 Ann. of Math. \textbf{129}  (1989),  113-149

\bibitem{Pho-Ste}
D.H. Phong \& E.M. Stein,
\emph{Estimates for
the Bergman and Szeg\" o projections on strongly pseudoconvex domains},
Duke Math. J. \textbf{44} (1977),  695-704


\bibitem{Zhu}
Kehe Zhu,
\emph{A sharp norm estimate for
the Bergman projection on $L^p$ spaces},  Bergman spaces and related topics in complex analysis, Contemp. Math
 \textbf{404} (2006),  199-205

\end{thebibliography}
\end{document}